\documentclass{article}
 \usepackage{tikz}
\usepackage{amssymb}
\usepackage{amsmath}
\usepackage{amsthm}
\usepackage{graphicx}
\usepackage{tikz}
\usepackage{pgfplots}
\pgfplotsset{compat=1.18} 
\usepackage{xcolor}
\usepackage[italicdiff]{physics}
\title{ }

\newtheorem{thm}{Theorem}
\newtheorem{prop}{Proposition}

\newtheorem{lem}{Lemma}

\begin{document}

\title
{ \textbf{ Contraction of   Convex Hypersurfaces in $\mathbb R^3$ by Powers of Principal Curvatures }  }
\author{Meraj Hosseini \\ {\small Department of Mathematics and Statistics, Concordia University} \\ {\small1455 Blvd. de Maisonneuve West, Montreal, QC, H3G 1M8, Canada} \\ {\small{\em meraj.hy@gmail.com}}}
\date{}

\maketitle \begin{abstract}
    We study the contraction of    strictly convex,  axially symmetric hypersurfaces by a non-symmetric, non-homogeneous, fully nonlinear function of curvature. Starting from   axially symmetric hypersurfaces with even profile curves,  we show  evolving hypersurfaces converge to a point in a finite time, and under proper rescaling, solutions will converge to  a convex hypersurface.
    
\end{abstract}

\def\thefootnote{ \ }
\footnotetext{{\em} $2010$ Mathematics Subject Classification: 53E99, 52A10, 53A04.
\par {\bf Keywords:} Non-symmetric and  non-homogeneous curvature flow, Convex hypersurfaces  }

 \section{Introduction}
 It was Mullins who first found solutions to the "curve shortening" flow \cite{MW}. 
  Within   decades different varieties of generalizations and studies
   were done in the plane by  many authors, most notably  by Gage and Hamilton  \cite{GH}, \cite{GA}, 
  \cite{GL},  by Andrews  \cite{BA},\cite{A5}, \cite{A7}, as well as many others    \cite{A}, \cite{GA}, \cite{GL}, \cite{ST}, \cite{AS4}, \cite{ASU}. We refer the reader to  \cite{Chou} for a more comprehensive account and list of references.
The mean curvature flow    is a well-known  generalization of the curve shortening flow to higher dimensions.  Huisken showed that the mean curvature flow converges to a round point in finite time \cite{H}.  The Gaussian curvature flow,  another generalization of the curve shortening flow to higher dimensions, 
 described by the following partial differential equation  was considered by Tso \cite{T}
 \begin{eqnarray}
   \pdv{X(x,t)}{t}= - K(x,t) \nu(x,t),& \textrm{  } & 
 \end{eqnarray}
 where $X$ is an embedding of a smooth, strictly convex hypersurface in $\mathbb{R}^n$ and $K$ is its Gaussian curvature.  
 Starting from a convex hypersurface, Tso \cite{T} showed that Gaussian curvature flow converges to a point in finite time, and rescaled solutions   converge to a convex hypersurface.  The following generalization of the Gaussian curvature flow was first studied by Chow \cite{C}
 \begin{eqnarray}\label{2}
   \pdv{X(x,t)}{t}= - K^{\beta}(x,t) \nu(x,t),& \textrm{  } & 
 \end{eqnarray}
 where $K$ is the Gaussian curvature, and $\beta>0$. Chow showed when $\beta= \frac{1}{n}$     rescaled solutions,  as in the case of the mean curvature flow, converge to a sphere.  When $\beta\neq \frac{1}{n}$, there is extensive literature studying the asymptotic behavior of the flow   (see \cite{BC} and references therein). 

 In this paper, we consider in $\mathbb{R}^3$ a version of the following generalization of the  flow  studied by Chow, equation (\ref{2}), which we will study in the $n$-space in a sequel of this paper:
 \begin{equation} \label{22}
    \pdv{X(x , t)}{t} =- \kappa_1^{\alpha_1}(x,t)  \kappa_2^{\alpha_2}(x,t)\dotsm\kappa_n^{\alpha_n}(x,t)\nu(x,t)
\end{equation}
 where $\kappa_1(x,t)\leq \kappa_2(x,t)\leq\dotsc\leq\kappa_n(x,t)$ are the principal curvatures and $\nu(x,t)$ is the outer normal to $K(t)$ at $X(x,t)$ and  $\alpha_1, \alpha_2,\dotsc,\alpha_n$ are positive real numbers. 
Contrary to  the  speed in the flow considered by Chow, Gaussian curvature,   the speed in the evolution equation (\ref{22})   is   non-homogeneous   and non-symmetric. 
Therefore, equation (\ref{22}) can be interpreted as  a more challenging generalization of the Gaussian  flow.
In the paper by Li and Lv   \cite{LL}  contraction of convex hypersurfaces by non-homogeneous speeds was studied for the first time. In their paper, the  speed functions, in addition to other properties, are  essentially required to be  symmetric and homogeneous of degree one.    Among other things, it was shown that such a flow evolves convex hypersurfaces to a point in finite time, and in fact, in hyperbolic space  roundness was proved  for sufficiently pinched initial hypersurfaces by high powers of the speed.  McCoy extended these results to hypersurfaces in Euclidean space. He  showed that, with sufficient initial curvature pinching,  the flow converges in finite time to points that are asymptotically spherical \cite{MJ}.  
For a review of the studied  non-homogeneous   flows in Euclidean space  
  see \cite{MJ1} and the references
therein. In all these works, in addition to other properties, the speed function is required to be
 symmetric and homogeneous of degree $1$, while in this paper the speed, $f(\kappa_1, \ldots, \kappa_n) =\kappa^{\alpha_1}_1 \ldots \kappa^{\alpha_n}_n  $,
  is neither
symmetric nor necessarily homogeneous of degree $1$.

We  restrict our attention to   the case in which  the initial surface,  $\partial K_0 $, viewed as the boundary of a smooth convex body $K_0$, is an axially symmetric surface   smoothly embedded in $  \mathbb{R}^3$.  
So, there exists a function $u_0:[0,1] \rightarrow \mathbb{R} $, meeting the $x$-axis orthogonally, strictly positive and smooth on $(0,1)$, with $u(0,0)=u(1,0)=0$,    such that $X_0:(0,1)\times \mathbb{S}^1 \rightarrow \mathbb{R}^3$, defined by $X_0(x,v)= (x,u_0(x)v)$,  parametrize  $\partial K_0$. We show that there exist   strictly positive, smooth functions  $u(\cdot,t) : (a_t,b_t) \rightarrow \mathbb{R}^{\geq 0}$, meeting the $x$-axis orthogonally,  with  $u(a_t,t)=u(b_t,t)=0$ satisfying  the following  equation, the scalar form of the  equation  (\ref{eq1}):
\begin{eqnarray} \label{Dr}
     \pdv{u(x,t)}{t}&=  -\frac{(-u_{xx})^{\alpha_2}}{u^{\alpha_1}(1+u_x^2)^{\frac{\alpha_1+3\alpha_2-1}{2}}} & \textrm{ on } (a_t,b_t)\times [0,\omega).
\end{eqnarray}  
 We show that, in finite time, solutions degenerate into a point, time at which curvatures develop a singularity, and,  under suitable initial conditions, the  solutions converge to a convex hypersurface. 
 {More precisely, suppose that the profile curve of the initial surface is even, that is, $u_0(x + p) = u_0(p - x)$ for every $x \in (c, b_0)$, where $p$ is the midpoint of the interval $(a_0, b_0)$. Then, we have the following theorem:
 \begin{thm}\label{thm1}
Suppose   smooth strictly convex embedded surface  $\partial K_0 \subset \mathbb R^3$   is axially symmetric  with an even profile curve. Then solutions of the equation (\ref{eq1})  exist on a maximal time interval $[0,\omega)$, and they  shrink  to a point as $t \longrightarrow \omega$. Furthermore, if solutions are rescaled to enclose  domains of constant volume,   they will converge sequentially in the Hausdorff metric to the boundary of a convex body.
 \end{thm}
\section{Preliminaries}
Let  a smooth, strictly  convex, axially symmetric surface $\partial K_0 \subset \mathbb R^3$, boundary of a convex body $K_0$,   be generated by revolving a strictly positive, smooth, concave function   $u: [0, a] \longrightarrow \mathbb R$ with $u(a)=u(0)=0$ about $x$-axis. 
Consider the parametrization $X:(0,a) \times [0,2\pi] \longrightarrow \mathbb{R}^3$   defined by
\begin{equation}\label{p1}
    X(x,\theta)=  \big(x, u(x) \cos\theta,u(x) \sin{\theta}\big).
\end{equation}
Then, for every $x \in (0,a)$ at  $\big(x,u(x)\big) \in \partial K$, the  principal curvatures are  equal  to
 \begin{eqnarray}
 \kappa_1(x)= \kappa_{rad}(x)=\frac{1}{u(x)\sqrt{1+u_x^2(x)}}, &  \kappa_2(x)=  \kappa_{axi}(x)=- \frac{u_{xx}(x)}{(1+u_x^2(x))^{\frac{3}{2}}}.
 \end{eqnarray}
We note that  since $K$ is smooth, $\kappa_{\min}, \kappa_{\max}$ are continuous and  poles are  umbilic points, then at poles we still have $\kappa_{rad}$ and $\kappa_{axi}$  defined by
  \begin{eqnarray}
  \kappa_1= \kappa_{rad}(p)= \lim_{x\longrightarrow p } \kappa_{rad}(x), \quad \text{and} \quad    \kappa_2=\kappa_{axi}(p)= \lim_{x\longrightarrow p } \kappa_{axi}(x),
\end{eqnarray}
 for $p \in \{ 0,a\}$.
In this paper, for $\alpha_1,\alpha_2 >0$ fixed, we consider   the following   flow 
 \begin{eqnarray}\label{eq1}
   \frac{\partial X}{\partial t} (p,t)=- \kappa^{\alpha_1}_{1}(p,t) \kappa^{\alpha_2}_{2}(p,t)\nu(p,t), && p \in \partial K.
\end{eqnarray}
We infer from equation (\ref{p1}) that for every $x \in (0,a)$,
$$\frac{\partial X}{\partial t}(x,\theta)=  (0, u_t \cos\theta,u_t \sin{\theta}).$$
If follows from
 $$\nu(x,\theta,t)=\frac{ (- u_x,\cos\theta,\sin \theta)}{\sqrt{1+u_x^2}}$$
 that
\begin{eqnarray}
    -  \kappa^{\alpha_1}_{rad}(x,t)\kappa^{\alpha_2}_{axi}(x,t)&=&\big< \frac{\partial X}{\partial t}(x,\theta), \nu(x,\theta,t)\big>  \\ &=& \big< (0, u_t \cos\theta,u_t \sin{\theta}), \frac{ (-u_x,\cos\theta,\sin \theta)}{\sqrt{1+u_x^2}}\big>=\frac{u_t}{\sqrt{1+u_x^2}}.  \nonumber
\end{eqnarray}
 Thus, the scalar evolution equation for the graph function $u$  is
 \begin{eqnarray}\label{sc1}
     u_t&=& -\sqrt{1+u_x^2}  \kappa^{\alpha_1}_{rad}(p,t)\kappa^{\alpha_2}_{axi}(x,t) \nonumber\\ &=& -\frac{(-u_{xx})^{\alpha_2} }{u^{\alpha_1}(1+u_x^2 )^{\frac{3\alpha_2+\alpha_1-1}{2}}}.
 \end{eqnarray}
Since the behavior of the flow around the poles is governed by similar equations, it is sufficient to study the behavior of the flow around one of the poles. More precisely, let $u_{\max}= u(x_0)$, and define $w:\big(-u(x_0),u(x_0)\big) \rightarrow [0,x_0)$ as follows:
\[   y \mapsto \left\{
\begin{array}{ll}
     u^{-1}(y)  & 0\leq y\leq  u_{\max}\\
     u^{-1}(-y)    & - u_{\max}\leq y< 0 
\end{array}. 
\right. \] 
About  $0=(0,0,0) \in \partial K$,   consider the    parametrization 
  $\bar X: B(0, u_{\max}) \longrightarrow \mathbb{R}^3$ defined by
  \begin{equation}\label{par}
     (y,z)\mapsto \big( v( y,z),y,z\big)   
  \end{equation}
   where $v(y,z) = w(\sqrt{y^2+z^2})= (u^{-1})(\sqrt{y^2+z^2})$.
It follows from
\begin{eqnarray}
    g^{ij}= \delta^{ij}- \frac{v_i v_j}{1+|\nabla v|^2}, \quad \text{and} \quad  h_{ij}= \frac{v^2_{ij}}{\sqrt{ 1+|\nabla v|^2}}
\end{eqnarray}
  that
  \begin{eqnarray}
 ( g^{ij})=\begin{bmatrix}
  \frac{1+v_z^2}{1+v_y^2+v_z^2} & -\frac{v_yv_z}{1+v_y^2+v_z^2}\\
 -\frac{v_yv_z}{1+v_y^2+v_z^2}& \frac{1+v_y^2}{1+v_y^2+v_z^2} 
  \end{bmatrix}, \quad \text{and} \quad
  (h_{ij})= \begin{bmatrix}
  \frac{v_{yy}}{\sqrt{1+v_y^2+v_z^2}} & \frac{ v_{yz}}{\sqrt{1+v_y^2+v_z^2}}\\
\frac{ v_{yz}}{\sqrt{1+v_y^2+v_z^2}}& \frac{v_{zz}}{\sqrt{1+v_y^2+v_z^2}}
  \end{bmatrix}.
  \end{eqnarray}
When $z=0$, we have that 
 \begin{eqnarray}
 ( g^{ij})=\begin{bmatrix}
   \frac{1}{1+v_y^2} & 0\\
 0& 1 
  \end{bmatrix}, \quad \text{and} \quad
  (h_{ij})= \begin{bmatrix}
  \frac{v_{yy}}{\sqrt{ 1+v_y^2}} & 0\\
0&\frac{  v_{zz}}{\sqrt{ 1+v_y^2}}
  \end{bmatrix}.
  \end{eqnarray}
As a result, along  the curve 
\begin{equation}\label{crv}
    \gamma=\Big\{ \big(y,v(y,0)\big): y \in \big(-u_{\max},u_{\max}\big)\Big\}
\end{equation}
the principal curvatures are
  \begin{eqnarray}
\kappa_{rad}(y,0)= \frac{v_{zz}}{(1+v_y^2)^{\frac{1}{2}}}, & \text{and} \quad   \kappa_{axi}(y,0)=  \frac{v_{yy}}{(1+v_y^2)^{\frac{3}{2}}},
\end{eqnarray}
    and   the  normal is 
     $$\nu(y,0)=\frac{(v_y,0,-1)}{\sqrt{v_y^2 +1}}.$$
     By direct calculations, we get
     $$(0,0,v_t(y,t))=\frac{\partial \gamma}{\partial t}(y,0,t)=- \kappa^{\alpha_1}_{rad}(y,0,t) \kappa^{\alpha_2}_{axi}(y,0,t) \nu(y,0,t).$$
     Therefore,
     \begin{equation} 
     \begin{split}
         & -  \kappa^{\alpha_1}_{rad}(y,t)\kappa^{\alpha_2}_{axi}(y,t)= \big< \frac{\partial \gamma}{\partial t}(y,0,t), \nu(y,\theta,t)\big>\\
         &  = \big< (0,0,v_t(y,t)) ,\frac{(v_y,0,-1)}{\sqrt{v_y^2 +1}}\big>=\frac{-v_t}{\sqrt{1+v_y^2}}. 
     \end{split}
 \end{equation}
So, the scalar evolution equation of the graph function  $v$  is
     \begin{equation}\label{bn}
       v_t =  \sqrt{1+v_y^2} \kappa^{\alpha_1}_{rad}(y,0,t) \kappa^{\alpha_2}_{axi}(y,0,t)=   \frac{1}{(1+v_y^2)^{\frac{\alpha_1+3\alpha_2-1}{2}}}   v^{\alpha_1}_{zz}v^{\alpha_2}_{yy}.
     \end{equation}
      We say smooth solutions for the equation (\ref{eq1}) exists if there exist smooth, strictly positive, concave functions $u(\cdot, t) : (a_t,b_t) \longrightarrow \mathbb R$ with $u(a_t,t)=u(b_t,t)=0$ satisfying equation (\ref{sc1}) such that for every $t$ the   curve $\gamma(t)$, defined by equation (\ref{crv}), is smooth.   
     \section{The Free Boundary Problem} 
  If we consider the inside and outside of the deforming hypersurface, starting with the initial hypersurface, as the interface separating two phases that are transforming, then the deformation of a hypersurface according to equation (\ref{eq1}) is  a free boundary problem. As in this problem, the only moving points are the points on the interface, the deforming hypersurface, the only non-zero PDE that appears when formulating  the problem is the equation (\ref{eq1}). 
Solutions of equation (\ref{Dr}) correspond to strictly convex hypersurfaces, the free boundaries,  moving according to  the equation (\ref{eq1}). 
In this section,  we show that solutions of equation (\ref{Dr})   exist, their convexity is preserved, and after a finite time, the surfaces of revolution described by the solutions converge to a point. We show that curvatures develop a singularity   when the volume enclosed by the solution becomes zero. Starting from a strictly convex, smooth surface of revolution the solutions always exist. We have 
\begin{lem}
Suppose that the initial conditions are the same as in Theorem \ref{thm1}. Then there exists $T>0$ such that solutions to the flow exist on $[0, T)$. 
\end{lem}
\begin{proof}
 We consider a linearization of the equation (\ref{sc1}).
Suppose that $ \epsilon >0$ is a small real number and  $\Tilde{u}=u+\epsilon \phi$ is a solution of the equation (\ref{sc1})  where $ \phi:[0,a]\times[0,\omega) \longrightarrow \mathbb{R}$ is a  smooth function.
Since
 $$\pdv{\Tilde{u} }{\epsilon}\Big|_{\epsilon=0} =\frac{\alpha_2(-u_{xx})^{\alpha_2 -1}}{u^{\alpha_1}(1+u_x^2)^{\beta}}\phi_{xx} + \frac{2\beta u_x}{u^{\alpha_1}(1+u_x^2)^{\beta+1}}\phi_x + \frac{\alpha_1}{u(1+u_x^2)^\beta}\phi$$
we infer from   $ \Tilde{u}_t=u_t+\epsilon \phi_t$ that 
$$\phi_t = \frac{\alpha_2(-u_{xx})^{\alpha_2 -1}}{u^{\alpha_1}(1+u_x^2)^{\beta}}\phi_{xx} + \frac{2\beta u_x}{u^{\alpha_1}(1+u_x^2)^{\beta+1}}\phi_x + \frac{\alpha_1}{u(1+u_x^2)^\beta}\phi.$$
Since $\frac{\alpha_2(-u_{xx})^{\alpha_2 -1}}{u^{\alpha_1}(1+u_x^2)^{\beta}}>0$ at time $t=0$ the equation is parabolic and  it  follows from classical theory of parabolic equations that solutions of equation (\ref{sc1}) exist. Now, we consider  the evolution equation of  
 $v_0:\big(-\max u(0), \max u(0)\big)\longrightarrow \mathbb R: $
 
  \begin{equation}\label{sc}
       v_t =  \frac{1}{(1+v_y^2)^{\frac{\alpha_1+3\alpha_2-1}{2}}}   v^{\alpha_1}_{zz}v^{\alpha_2}_{yy}.
     \end{equation}
  Suppose that $ \epsilon >0$ is a small real number and  $\Tilde{v}=v+\epsilon \phi$ is a solution of the equation (\ref{sc})   where $\phi$ is  a smooth function.
 Since
\begin{equation*}
\begin{split}
     \pdv{\Tilde{v}_t}{\epsilon} &=-\big(\frac{3\alpha_2+\alpha_1-1}{2}\big) \frac{ 4(v_y+\epsilon \phi_y)}{(1+2(v_y+\epsilon \phi_y)^2)^{\frac{3\alpha_2+\alpha_1+1}{2}}} (v_{yy}+\epsilon \phi_{yy})^{\alpha_2} (y,t) v^{\alpha_1}_{zz}(y,t)\phi_y\\
     &+  \alpha_2 \frac{ 1}{(1+2(v_y+\epsilon \phi_y)^2)^{\frac{3\alpha_2+\alpha_1-1}{2}}}(v_{yy}+\epsilon \phi_{yy})^{\alpha_2-1} (y,t) v^{\alpha_1}_{zz}(y,t) \phi_{yy}
\end{split}
\end{equation*}
it follows from $\Tilde{v}_t=v_t+\epsilon \phi_t$ that
\begin{equation}\label{lz}
\begin{split}
 \phi_t&=     \frac{ \alpha_2(v_{yy})^{\alpha_2-1}  v^{\alpha_1}_{zz} }{(1+2(v_y )^2)^{\frac{3\alpha_2+\alpha_1-1}{2}}} \phi_{yy} 
 -\big(\frac{3\alpha_2+\alpha_1-1}{2}\big) \frac{ 4v_y (v_{yy}  )^{\alpha_2}   v^{\alpha_1}_{zz}}{(1+2(v_y )^2)^{\frac{3\alpha_2+\alpha_1+1}{2}}}  \phi_y.  
\end{split}
\end{equation}
Since equation (\ref{lz}) is parabolic,  solutions to the equation (\ref{sc}) exist.
Since the solution of equation (\ref{lz}) satisfies equation (\ref{sc1}) on its domain,
we infer from the uniqueness of solutions that the solutions of equation (\ref{sc1}) and equation (\ref{sc}) are the same on their common domain. We infer that solutions of equation (\ref{eq1}) exist.
\end{proof}
 \begin{lem}\label{ll}
Under the initial conditions specified in the Theorem \ref{thm1}, the flow preserves the convexity of the evolving surfaces as long as the solutions exist.
\end{lem}
\begin{proof}
We set $H=u_t$ and differentiate  equation (\ref{sc1}) with respect to time  to get the evolution equation of the speed:
\begin{equation} \label{H}
\begin{split}
  H_t &=  \frac{\alpha_2(-u_{xx})^{\alpha_2-1}u^{\alpha_1}(1+u_x^2)^{\beta}}{u^{2\alpha_1}(1+u_x^2)^{2\beta}}H_{xx} + \frac{2\beta u_xu^{\alpha_1}(-u_{xx})^{\alpha_2}(1+u_x^2)^{\beta-1}}{u^{2\alpha_1}(1+u_x^2)^{2\beta}}H_{x}\\
  &+ \frac{\alpha_1 u^{\alpha_{1}-1}(-u_{xx})^{\alpha_2}(1+u_x^2)^{\beta}}{u^{2\alpha_1}(1+u_x^2)^{2\beta}}H. 
\end{split}
\end{equation}
 We choose $r >0$ to be big enough such that
    \begin{equation}\label{I}
        \kappa_1^{\alpha_1}\kappa_2^{\alpha_2}(x, 0) > \big(\frac{1}{r}\big)^{\alpha_1+\alpha_2}. 
    \end{equation}
    Suppose that the maximum of $u(\cdot,0)$  happens at $q$ and consider the sphere generated by revolving $(x-q)^2+ y^2=r^2$ around $x$-axis. By direct calculation, we get 
    $$y_x = \frac{-(x-q)}{\sqrt{r^2-(x-q)^2}}.$$
    Since maximum of $u$ happens at $q$, $u_x(q)=0$. Since $u_x$ and $ y_x$ are continuous and $u_x(q)=y_x(q)=0$ we infer from  equation (\ref{I}) that there exist $\epsilon>0$ such that for every $ x \in (q-\epsilon,q+\epsilon)$ we have 
       \begin{equation}\label{II}
      \sqrt{1+u_x^2}  \kappa_1^{\alpha_1}\kappa_2^{\alpha_2} > \sqrt{1+y_x^2}\big(\frac{1}{r}\big)^{\alpha_1+\alpha_2}. 
    \end{equation}
    Now, let small $\delta >0$   be arbitrary, and  let $$0<e = \min\limits_{x \in [q+\epsilon,a-\delta]}\sqrt{1+u_x^2}.$$
    For every $x \in [ q+\epsilon, a-\delta]$, since $x \leq a-\delta$ we infer that
    \begin{equation}
      \frac{ (a-\delta-q)^2}{ r^2-(a-\delta-q)^2 }  \geq \frac{ (x-q)^2}{ r^2-(a-\delta-q)^2 }  \geq y^2_x = \frac{ (x-q)^2}{ r^2-(x-q)^2 }.
    \end{equation}
   If we choose $r>0$ to be big enough, then for every $x \in [ q+\epsilon, a-\delta] $
    \begin{equation}\label{III}
     u_x^2\geq     \frac{ (a-\delta-q)^2}{ r^2-(a-\delta-q)^2 }  \geq \frac{ (x-q)^2}{ r^2-(a-\delta-q)^2 }  \geq y^2_x = \frac{ (x-q)^2}{ r^2-(x-q)^2 }.
    \end{equation}
  Since $u$ is even  we infer from     equations (\ref{I}) and (\ref{III}) that    for every $\delta >0$ small enough  there exist $r>0$ big enough such that for every $x \in [ \delta, a-\delta]$ we have 
     \begin{equation}\label{IIII}
       -H(x,0) =  \sqrt{1+u_x^2} \kappa^{\alpha_1}_{rad}\kappa^{\alpha_2}_{axi}(x,0) 
   \geq - y_t(x,0) 
    = \sqrt{1+y_x^2}(\frac{1}{r})^{\alpha_2+\alpha_1}.
    \end{equation}
  Let 
$u(\cdot,t):(a_t,b_t)\longrightarrow \mathbb R$ and $y(\cdot,t):(c_t,d_t)\longrightarrow \mathbb R$  be the   solutions corresponding respectively to evolving surfaces and shrinking balls. Since supersolutions dominate subsolutions, the equation (\ref{IIII}) will be  preserved on $[0,t]$ for every $t \in [0,\omega)$. Therefore, for every $t \in [0,\omega)$,  and every $x \in (a_t ,b_t)$, we have 
$$\kappa_{axi}(x,t), \kappa_{rad}(x,t) >0.$$
 We apply the same argument to the poles. Namely,
   consider the curve $v_0$. As we have seen the principal curvatures along $v(y,0)=v_0$ are
  \begin{eqnarray}
\kappa_1=\kappa_{rad}(y,0)= \frac{v_{zz}}{(1+v_y^2)^{\frac{1}{2}}}, & \text{and} \quad  \kappa_2= \kappa_{axi}(y,0)=  \frac{v_{yy}}{(1+v_y^2)^{\frac{3}{2}}}.
\end{eqnarray}
 Since minimum of $v_0$ happens at $0$, $v_y(0,0)=0$. Let $r>0$ be big enough that equation (\ref{I}) holds, and consider the sphere generated by revolving $(x-c)^2+y^2=r^2$ around $x$-axis.  We note that
 $$x_y = \frac{-y}{\sqrt{ r^2-y^2}}.$$
 Since $v_y ,x_y$ are continuous and $v_y(0,0)=x_y(0,0)=0$ we infer   from the equation (\ref{I}) that there exist $\epsilon>0$ such that for every $ y \in ( -\epsilon,\epsilon)$ we have 
       \begin{equation}\label{IIIII}
      \sqrt{1+v_y^2}  \kappa_1^{\alpha_1}\kappa_2^{\alpha_2} > \sqrt{1+x_y^2}\big(\frac{1}{r}\big)^{\alpha_1+\alpha_2}. 
    \end{equation}
 Let $t \in [0,\omega)$ be arbitrary, $\epsilon >0$  be small enough.  As we have seen, equation (\ref{bn}) gives the evolution equation of $v_0$.  We  find the evolution equation of the speed $v_t$. Consider 
 $ \eta :[ -\epsilon, \epsilon] \times [0,t]\longrightarrow \mathbb{R}   $   defined by 
 \begin{equation} \label{eta}
    \eta:=   v_t =  \frac{1}{(1+v_y^2)^{\frac{\alpha_1+3\alpha_2-1}{2}}}   v^{\alpha_1}_{zz}v^{\alpha_2}_{yy}.
     \end{equation}
We set $\frac{\alpha_1+3\alpha_2-1}{2}= \mu$, and differentiate $\eta$ with respect to time to get the evolution equation for the speed of the evolving curves $v(y,t)$:
\begin{equation}\label{ve}
\begin{split}
  \eta_t(y,t)&= \Big[ \frac{-2\mu v_y  v^{\alpha_1}_{zz}v^{\alpha_2}_{yy}}{(1+v_y^2)^{\mu+1} }\Big] \eta_y+ \Big[ \frac{\alpha_1  v^{\alpha_2}_{yy}v^{\alpha_1-1}_{zz}}{(1+v_y^2)^{\mu} }\Big] \eta_{zz}+   \Big[ \frac{\alpha_2  v^{\alpha_2-1}_{yy}v^{\alpha_1}_{zz}}{(1+v_y^2)^{\mu} }\Big] \eta_{yy}\\
  &=     \Big[ \frac{\alpha_2  v^{\alpha_2-1}_{yy}v^{\alpha_1}_{zz}}{(1+v_y^2)^{\mu} }\Big] \eta_{yy} +\Big[ \frac{-2\mu v_y  v^{\alpha_1}_{zz}v^{\alpha_2}_{yy}}{(1+v_y^2)^{\mu+1} }\Big] \eta_y+  \frac{ \alpha_1 \eta_{zz}}{v_{zz}} \eta.
 \end{split}
\end{equation}
  So, the evolution equation of $\eta$ is   parabolic with uniformly bounded continuous coefficients on $[-\epsilon,\epsilon] \times [0,t]$. 
 Since supersolutions dominate subsolutions equation (\ref{IIIII}) holds on $[-\epsilon,\epsilon] \times [0,t]$. In turn, the principal curvatures remain positive. 
 \end{proof}
 \section{Singularity}
   We infer from Blaschke selection theorem that at final time $\omega$ there exists a convex set $K(\omega)$  that solutions converge to.  In this section,  we first show the volume enclosed by   $K(\omega)$  is zero. Then, we show,  actually, solutions converge to a single point as $t \longrightarrow \omega$.
\begin{prop}
Suppose that the initial conditions are the same as in theorem \ref{thm1}, and suppose that $[0,\omega)$  is the maximal time interval on which the solutions exist.  Then
\begin{eqnarray*}
 V(t) \rightarrow 0 \textrm{ as } t \rightarrow \omega.
\end{eqnarray*}
 \end{prop}
  \begin{proof}
Suppose for contradiction that $V(\omega) \neq 0$. Therefore, there exist   $[a,b]$ with $a<b$  such  that $[a,b] \subset (a_t,b_t)$ for every $t \in [0,\omega)$. Suppose that $[a,b]$  is the longest  interval with this property, and  
  suppose $K(\omega)$ is the convex body with the profile curve $u(x,\omega)= \inf\limits_{t \in [0,\omega)} u(x,t)$.  Since $u(x,\omega)$ is concave, it is differentiable except at finitely many points and second differentiable almost everywhere.  Therefore, there exist $x_0 \in (a,b)$ such that  
\begin{eqnarray}
 \kappa_{axi}(x_0,\omega):=\sup\limits_{t \in [0,\omega)} \kappa_{axi}(x_0,t)=  \sup\limits_{t \in [0,\omega)} \frac{(-u_{xx})  }{ (1+u_x^2 )^{\frac{3 }{2 }}}<\infty
\end{eqnarray}
 Since $u(x_0,t)$ is bounded from below
\begin{equation}
  \sup\limits_{t \in [0,\omega)}\kappa_{rad}(x_0,t)= \frac{1  }{ u(1+u_x^2 )^{\frac{1 }{2 }}}<\frac{1  }{ u }\leq\frac{1}{\epsilon}<\infty 
\end{equation}
for some $\epsilon >0$. Since $u^2_x(x_0, \omega)$  exist we infer from   the last two inequalities that for some $C>0$
\begin{equation}\label{31}
 \sup\limits_{t \in [0,\omega)} -u_t(x_0,t)= \sup\limits_{t \in [0,\omega)}\sqrt{1+u^2_x}     \kappa^\alpha_{axi}(x_0,t)\kappa^{\beta}_{rad}(x_0,t) < C.
\end{equation}
We distinguish two parts:
\begin{itemize}
    \item[Part 1:] Since $u(\cdot, \omega)$ is even, there exist two points, say $x_0$ and $x_1$, with the same speed,  around the midpoint of $(a, b)$ such that the equation (\ref{31}) holds at these points. For every $t \in [0,\omega)$, consider $H:[x_0,x_1]\times [0,t] \longrightarrow \mathbb R$ defined in previous lemma, $H=u_t.$ Since the evolution equation of $H$, the equation (\ref{H}),  is parabolic with uniformly bounded coefficients, it follows from maximum principle that  
    \begin{equation}
       \max\limits_{[x_0,x_1]\times [0,t]} -u_t \leq \max \big\{ \max\limits_{t\in [0,\omega)}-u_t(x_0,t), \max\limits_{[x_0,x_1]}-u_t(x,0)\big\} < \infty.
    \end{equation}
    So, the speed $u_t:[x_0,x_1]\times [0,\omega) \longrightarrow \mathbb R$ is bounded on   $[x_0,x_1]$.  We claim that the principal curvatures remain bounded from above on $[x_0,x_1]$. To see  this, we note that the lower bound on $u$ implies an upper bound on the radial curvature. If the axial curvature becomes infinite at a certain point, the boundedness of the product of curvatures implies the radial curvature tends toward zero at that point. The only way for the radial curvature to approach zero is if the derivative at that point   becomes infinite. However, this is impossible as $u(\cdot,\omega)$ is concave, and even if the derivative does not exist, it cannot be infinite.
 \item[Part 2:]
For every $t \in [0,\omega)$, consider     $\eta :[-c,c] \times [0,t] \longrightarrow \mathbb R$ defined by equation (\ref{eta})  where $$0< c =   u(x_0,\omega). $$   Since for every $t \in [0,\omega)$, the evolution equation of $\eta$, the equation (\ref{ve}), is   parabolic with uniformly bounded coefficients on $ [-c,c] \times [0,t]$, we conclude that 
   \begin{equation}
       \max\limits_{[-c,c]\times [0,t]} \eta  \leq \big\{ \max\limits_{t\in [0,\omega)}\eta (c,t), \max\limits_{y\in [-c,c]} \eta(y,0)\big\} < \infty.
    \end{equation}
     So, the speed $\eta:[-c,c] \times [0,\omega) \longrightarrow R$ is bounded on $[-c,c]$. Consequently, the product of curvatures at $y=0$, which is an umbilic point and corresponds to $x=a$, remains bounded.
\end{itemize}
Since the choice of $x_0$ and $ x_1$ is flexible, we infer from Part 1 and Part 2 that the principal curvatures of $K(\omega)$ are bounded from above, and in turn, the flow can be continued beyond $\omega$. This is impossible as $[0,\omega)$ is the maximal time interval on which solutions exist. So, as the final time is approached,  the volume of the domain enclosed by the solutions tends to zero.
 \end{proof}
 \begin{prop}
     The flow converges to a single point.
 \end{prop}
 \begin{proof}
 To see this, suppose not, so $u(x,t)$ will degenerate into a segment. Since $K(t)$ is symmetric with respect to the $x$-axis, this segment will either lie on the $x$-axis or be parallel to the $y$-axis.
Consider the first case where the flow  degenerates into a segment on the $x$-axis, say to $[a,b]$ with $a<b$. 
 By performing direct calculations, we find the evolution equation of $u_{xx}$:
\begin{equation}
    \begin{split}
    u_{xt }& =   \frac{\alpha_2(-u_{xx})^{\alpha_2-1}u_{xxx}}{u^{\alpha_1}(1+u_x^2)^{\beta}} + \frac{\left[\alpha_1 u_x u^{\alpha_1-1}(1+u_x^2)^{\beta}+2\beta u_xu_{xx}u^{\alpha_1}(1+u_x^2)^{\beta-1}\right](-u_{xx})^{\alpha_2}}{u^{2\alpha_1}(1+u_x^2)^{2\beta}}, \\
    & =   \frac{\alpha_2(-u_{xx})^{\alpha_2-1}}{u^{\alpha_1}(1+u_x^2)^{\beta}}u_{xxx}+\frac{ 2\beta u_x  (-u_{xx})^{\alpha_2}}{u^{\alpha_1}(1+u_x^2)^{\beta+1}}u_{xx}  + \frac{ \alpha_1   (-u_{xx})^{\alpha_2}}{u^{\alpha_1+1}(1+u_x^2)^{\beta}}  u_x \\
    &
\end{split}
\end{equation}   
and, in turn,
\begin{equation}\label{lng}
    \begin{split}
    u_{xxt } &=  \frac{\alpha_2(-u_{xx})^{\alpha_2-1}}{u^{\alpha_1}(1+u_x^2)^{\beta}} u_{xxxx}+ \frac{  2\beta u_x   (-u_{xx})^{\alpha_2}}{u^{\alpha_1}(1+u_x^2)^{\beta+1} }u_{xxx} +\frac{ \alpha_1    (-u_{xx})^{\alpha_2}}{u^{\alpha_1+1}(1+u_x^2)^{\beta}}u_{xx} \\
    & -  \frac{(\alpha_2-1)\alpha_2(-u_{xx})^{\alpha_2-2}}{u^{\alpha_1} (1+u_x^2)^{\beta}} u^2_{xxx} - \frac{\alpha_1\alpha_2u_x(-u_{xx})^{\alpha_2-1}}{u^{\alpha_1+1}(1+u_x^2)^{\beta}   } u_{xxx}+ \frac{2\alpha_2 \beta u_x(-u_{xx})^{\alpha_2}}{u^{\alpha_1}(1+u_x^2)^{\beta+1}  } u_{xxx}\\
    &- \frac{  2\beta     (-u_{xx})^{\alpha_2+1}}{u^{\alpha_1+1}(1+u_x^2)^{\beta} }u_{xx}  - \frac{  2\beta \alpha_2 u_x  u_{xxx} (-u_{xx})^{\alpha_2-1}}{u^{\alpha_1+1}(1+u_x^2)^{\beta} }u_{xx} + \frac{  2\beta u_x   (-u_{xx})^{\alpha_2}}{u^{\alpha_1+1}(1+u_x^2)^{\beta} }u_{xx}  \\
    &+   \frac{  2\beta(\alpha_1+1) u^2_x   (-u_{xx})^{\alpha_2}}{u^{\alpha_1+2}(1+u_x^2)^{\beta} }u_{xx}- \frac{  4\beta^2 u^2_x   (-u_{xx})^{\alpha_2+1}}{u^{\alpha_1+1}(1+u_x^2)^{\beta+1} }u_{xx}+\frac{ \alpha_1  \alpha_2  (-u_{xx})^{\alpha_2-1}u_{xxx}}{u^{\alpha_1}(1+u_x^2)^{\beta+1}}u_{x}\\
    & -\frac{ \alpha^2_1    (-u_{xx})^{\alpha_2}}{u^{\alpha_1+1}(1+u_x^2)^{\beta+1}}u^2_{x}+\frac{ 2\alpha_1 (\beta+1)   (-u_{xx})^{\alpha_2+1}}{u^{\alpha_1}(1+u_x^2)^{\beta+2}}u^2_{x}.
\end{split}
\end{equation}
Let $v=u_{xx}$. Then
\begin{equation} 
    \begin{split}
    v_{t } &=  \frac{\alpha_2(-u_{xx})^{\alpha_2-1}}{u^{\alpha_1}(1+u_x^2)^{\beta}} v_{xx}+ \frac{  2\beta u_x   (-u_{xx})^{\alpha_2}}{u^{\alpha_1}(1+u_x^2)^{\beta+1} }v_{x} +\frac{ \alpha_1    (-u_{xx})^{\alpha_2}}{u^{\alpha_1+1}(1+u_x^2)^{\beta}}v  \\
    &+\Big[ -  \frac{(\alpha_2-1)\alpha_2(-u_{xx})^{\alpha_2-2}u_{xxx}}{u^{\alpha_1} (1+u_x^2)^{\beta}}   - \frac{\alpha_1\alpha_2u_x(-u_{xx})^{\alpha_2-1}}{u^{\alpha_1+1}(1+u_x^2)^{\beta}   } + \frac{2\alpha_2 \beta u_x(-u_{xx})^{\alpha_2}}{u^{\alpha_1}(1+u_x^2)^{\beta+1}  } \Big]v_{x}\\
    &+\Big[- \frac{  2\beta     (-u_{xx})^{\alpha_2+1}}{u^{\alpha_1+1}(1+u_x^2)^{\beta} }   - \frac{  2\beta \alpha_2 u_x  u_{xxx} (-u_{xx})^{\alpha_2-1}}{u^{\alpha_1+1}(1+u_x^2)^{\beta} }  + \frac{  2\beta u_x   (-u_{xx})^{\alpha_2}}{u^{\alpha_1+1}(1+u_x^2)^{\beta} }\Big]v  \\
    &+   \frac{  2\beta(\alpha_1+1) u^2_x   (-u_{xx})^{\alpha_2}}{u^{\alpha_1+2}(1+u_x^2)^{\beta} }v- \frac{  4\beta^2 u^2_x   (-u_{xx})^{\alpha_2+1}}{u^{\alpha_1+1}(1+u_x^2)^{\beta+1} }v+\frac{ \alpha_1  \alpha_2  (-u_{xx})^{\alpha_2-1}u_{x}}{u^{\alpha_1}(1+u_x^2)^{\beta+1}}v_{x}\\
    & +\frac{ \alpha^2_1   u^2_{x} (-u_{xx})^{\alpha_2-1}}{u^{\alpha_1+1}(1+u_x^2)^{\beta+1}}v-\frac{ 2\alpha_1 (\beta+1) u^2_{x}  (-u_{xx})^{\alpha_2}}{u^{\alpha_1}(1+u_x^2)^{\beta+2}}v
\end{split}
\end{equation}
Since $u(x,t)$ degenerates into a segment, at every point of $(a,b)$, say at $x_0 \in (a,b)$, $u_{xx}$ tends  to zero, and at endpoints $u_{xx}$ tends to negative infinity. For every $t \in [0,\omega)$,  consider $u_{xx}:\Omega_t:=[a,b]\times [0,t] \longrightarrow \mathbb R$, and let $\Gamma_t$ be the parabolic boundary of $\Omega_t$. We have
\begin{equation}\label{iam}
 u_{xx}(x_0,t) \leq \max\limits_{\Gamma_t}u_{xx}(x,t).   
\end{equation}
As $t \longrightarrow \omega$, the right-hand side of the equation (\ref{iam}) tends to either a negative number or negative infinity while the left-hand side tends to zero. So, this case cannot occur. 
Consider the case where the flow degenerates into a segment parallel to the $y$-axis, say to $[-c,c]\times \{\bar x\}$. We define the function  $\eta:\Sigma_t:=[-c,c] \times [t_0,t]\longrightarrow \mathbb R$  by
\begin{equation} 
    \eta=   v_t =  \frac{1}{(1+v_y^2)^{\frac{\alpha_1+3\alpha_2-1}{2}}}   v^{\alpha_1}_{zz}v^{\alpha_2}_{yy}.
     \end{equation}
 The evolution equation of   $\eta$, equation (\ref{ve}),   is parabolic with uniformly bounded continuous coefficients on $\Sigma_t$:
\begin{equation*}
\begin{split}
  \eta_t(y,t)&= \Big[ \frac{-2\mu v_y  v^{\alpha_1}_{zz}v^{\alpha_2}_{yy}}{(1+v_y^2)^{\mu+1} }\Big] \eta_y+ \Big[ \frac{\alpha_1  v^{\alpha_2}_{yy}v^{\alpha_1-1}_{zz}}{(1+v_y^2)^{\mu} }\Big] \eta_{zz}+   \Big[ \frac{\alpha_2  v^{\alpha_2-1}_{yy}v^{\alpha_1}_{zz}}{(1+v_y^2)^{\mu} }\Big] \eta_{yy}\\
  &=     \Big[ \frac{\alpha_2  v^{\alpha_2-1}_{yy}v^{\alpha_1}_{zz}}{(1+v_y^2)^{\mu} }\Big] \eta_{yy} +\Big[ \frac{-2\mu v_y  v^{\alpha_1}_{zz}v^{\alpha_2}_{yy}}{(1+v_y^2)^{\mu+1} }\Big] \eta_y+  \frac{ \alpha_1 \eta_{zz}}{v_{zz}} \eta.
 \end{split}
 \end{equation*}
At $y=0$,  solutions are umbilic, so both principal curvatures and, in turn, the speed tends to zero at this point. Since at the endpoints, the speed tends to infinity an argument similar to the one in the previous part shows this is impossible. Therefore, the solutions cannot degenerate into segments. Consequently,  it tend to a point as $t \longrightarrow \omega$.
 \end{proof}    
  \section{Asymptotic behavior of the flow}
While it seems plausible to prove the previous results even when relaxing the condition that the initial profile curve is even,  when studying the asymptotic behavior of the flow, this condition plays a crucial role. 
 We demonstrate that by properly rescaling the solutions, starting from an even initial data, rescaled solutions converge to a convex hypersurface.
  \begin{prop}
Suppose the initial conditions are the same as specified in Theorem \ref{thm1}. Then, for every sequence of times $\{t_n\}\nearrow \omega$, the properly rescaled solutions possess a subsequence that converges to a convex body.
\end{prop}
\begin{proof}
 Since the   profile curve of the initial surface is contracting with the same speed at points $x$ equally distanced from the midpoint of $(a_0,b_0)$,  for every $t \in[0,\omega)$, the profile curve corresponding to the solution at time $t$ is even.    We note that for every $t \in [0,\omega)$, there exist $x_t \in (a_t,b_t)$ such that  $$ V(t)=\pi\int\limits_{a_t}^{b_t} u^2(x,t) \dd x = \pi (b_t-a_t)u^2(x_t,t).$$
     Let $p$ be the point on the $x$-axis to which   solutions converge.    
   Consider 
   $$   I_t=(c_t,d_t):=\left(\frac{\pi(a_t-p)u^2(x_t,t)}{{V(t)}}, \frac{\pi(b_t-p)u^2(x_t,t)}{{V(t)}}\right)= (\frac{a_t-p}{b_t-a_t},\frac{b_t-p}{b_t-a_t}),$$
   and define  $T:(a_t,b_t) \longrightarrow I_t$ by 
   $$x \mapsto \frac{\pi(x-p)u^2(x_t,t)}{{V(t)}}=\frac{x-p}{b_t-a_t}.$$
   Since solutions are even, the speed at the endpoints are the same but of the opposite signs,  
      $0\neq  \pdv{a_t}{t}=-  \pdv{b_t}{t}$. In turn,
      \begin{eqnarray}
          \lim\limits_{t\longrightarrow \omega}\frac{a_t-p}{b_t-a_t}=-\frac{1}{2}, &\text{ and } \lim\limits_{t\longrightarrow T}\frac{b_t-p}{b_t-a_t}=\frac{1}{2}.
      \end{eqnarray}
   As a result  $I_t\longrightarrow(-\frac{1}{2},\frac{1}{2})$ as   $t \longrightarrow \omega$. 
  Consider the rescaling   $\widetilde{ u}:I_t \longrightarrow \mathbb{R}^{\geq0} $ defined by
  \begin{equation}
      y \mapsto \sqrt\frac{\pi(b_t-a_t)}{V(t)}u(\frac{yV(t)}{\pi u^2(x_t,t)}+p,t)
  \end{equation}
   We note that   rescaled solutions enclose a domain of  constant volume. More precisely,  
   $$\widetilde{V}(t)= \pi \int\limits_{c_t}^{d_t} \frac{\pi(b_t-a_t)}{V(t)}u^2(\frac{yV(t)}{\pi u^2(x_t,t)}+p,t)\dd y= \frac{\pi(b_t-a_t)}{V(t)}\pi \int\limits_{a_t}^{b_t} u^2(x,t) \frac{\pi u^2(x_t,t)}{V(t)} \dd x$$
   $$ =  \frac{\pi u^2(x_t,t)(b_t-a_t)}{V(t)} \pi =\pi. $$ 
If the rescaled solutions degenerate, then since $I_t \longrightarrow (-\frac{1}{2},\frac{1}{2})$, they will degenerate into the segment $(-\frac{1}{2},\frac{1}{2})$. However, this contradicts the fact that the rescaled solutions enclose constant volume $\pi$. In addition, since  $I_t \longrightarrow (-\frac{1}{2},\frac{1}{2})$ and volume is constant,  the rescaled solutions are bounded from above. Therefore, rescaled solutions  are included in a compact annulus, and it follows from Blaschke selection theorem that   $\{\widetilde{K}(t_n)\}$ is subsequentially  convergent to a convex body $\widetilde K$.
\end{proof}

\section*{Acknowledgments}
I would like to thank  Alina Stancu for her  comments and encouragement. This work would not have been possible without her  guidance and support.\\ 
\\
 
\end{document}